\documentclass[11pt]{article}
\usepackage{amsmath}
\usepackage{amsthm}
\usepackage{amsfonts} 
\usepackage[english]{babel} 
\usepackage[colorlinks = true,
            linkcolor = blue,
            urlcolor  = blue,
            citecolor = blue,
            anchorcolor = blue]{hyperref} 
\usepackage{dsfont}
\usepackage{enumitem}

\usepackage{physics}
\allowdisplaybreaks
\usepackage[a4paper, lmargin=0.1666\paperwidth, rmargin=0.1666\paperwidth, tmargin=0.1111\paperheight, bmargin=0.1111\paperheight]{geometry} 
\hypersetup{ 	
	pdfsubject = {},
	pdftitle = {},
	pdfauthor = {}
}

\theoremstyle{plain}
\newtheorem{theorem}{Theorem}[section]

\newtheorem{lemma}[theorem]{Lemma}

\theoremstyle{definition}

\theoremstyle{remark}

\newtheoremstyle{named}{}{}{\itshape}{}{\bfseries}{.}{.5em}{\thmnote{#3}#1}
\theoremstyle{named}

\title{The number of correct guesses with partial feedback}

\author{Zipei Nie \thanks{Lagrange Mathematics and Computing Research Center, Huawei. Email: niezipei@huawei.com.}}

\begin{document}

\maketitle
\begin{abstract}
We consider the following game. A deck with $m$ copies of each of $n$ distinct cards is shuffled in a perfectly random way. The Guesser sequentially guesses the cards from top to bottom. After each guess, the Guesser is informed whether the guess is correct. The goal is to maximize the expected number of correct guesses.

We prove that, if $n= \Omega(\sqrt{m})$, then at most $m+O(\sqrt{m})$ cards can be guessed correctly. Our result matches a lower bound of the maximal expected payoff by Diaconis, Graham and Spiro when $n=\Omega(m)$.
\end{abstract}
\section{Introduction}

For two integer parameters $m$ and $n$, the following game is called a card guessing game with partial feedback.

At the beginning, the Dealer holds a deck of $mn$ cards, which contains $m$ copies of each of $n$ distinct cards labeled $1,2,\ldots,n$. In each round, the Dealer draws a card uniformly at random from the deck, and the Guesser guesses this card, and then this card is discarded. After each guess, the Guesser is only told whether their guess is correct. The goal of the Guesser is to maximize the expected number of correct guesses after $mn$ rounds. Throughout the introduction, let $P(\mathcal{G})$ denote the number of correct guesses using a guessing strategy $\mathcal{G}$. 

This game has attracted much attention, as it provides multiple real-life applications. In 1957, Blackwell and Hodges \cite{blackwell1957design} (see also \cite{efron1971forcing}) used this game to model the clinic trials. In particular, they considered the case where $n=2$ and derived the optimal strategy. This special case is significantly easier than the general situation because the Guesser can determine the card type given the Yes/No feedback, that is, the game with partial feedback is equivalent to the game with complete feedback. In 2005, the same problem was analyzed independently by Ethier and Levin \cite{ethier2005fundamental} (see also \cite{ethier2010doctrine}) in the scenario of casino games. In 1978, Diaconis \cite{diaconis1978statistical} modeled some ESP (extrasensory perception) experiments as the card guessing game with partial feedback where $m=n$. In 1981, Diaconis and Graham \cite{diaconis1981analysis} studied this game systematically. They derived the optimal strategy for the case where $m=1$. 

Although this game has a long history, not much progress has been made until recently. Diaconis, Graham, He and Spiro proved the following result. In particular, they proved that the expected payoff has a constant upper bound when $m$ is a constant.   

\begin{theorem}\cite{diaconis2022card}\label{ref-theorem-1}
If $n$ is sufficiently large in terms of $m$, then
$$\mathbb{E}\left(P(\mathcal{G})\right)\le  m +c m^{3/4}\log^{1/4} m$$
for an absolute constant $c>0$.
\end{theorem}

Later, Diaconis, Graham and Spiro proved the following lower bound of the maximal expected payoff.

\begin{theorem}\cite{diaconis2022guessing}\label{ref-theorem-2}
If $n\ge 8m$, then there exists a guessing strategy $\mathcal{G}$ with 
$$\mathbb{E}\left(P(\mathcal{G})\right) \ge m + \frac{1}{40} \sqrt{m}.$$
\end{theorem}

Thus it is of interest (see \cite[Conjecture 4.3]{diaconis2022card} and \cite[Conjecture 3]{diaconis2022guessing}) to close the gap between Theorem \ref{ref-theorem-1} and Theorem \ref{ref-theorem-2}. This is the purpose of our paper. We are going to prove the following result.

\begin{theorem}\label{main-thm}
If $n\ge 1200\sqrt{m}$, then
$$\mathbb{E}\left(P(\mathcal{G})\right)\le  m +500\sqrt{m}.$$
\end{theorem}

We conjecture that the assumption $n\ge 1200\sqrt{m}$ could be dropped. The main reason we need this assumption in our proof is that, we take \cite[Lemma 3.1]{diaconis2022card} as the starting point of our proof. Intuitively, this step costs a multiplicative factor of $1+\frac{1}{n}$, which is only negligible when $n$ is at least of the order of $\sqrt{m}$. 

Note that the condition in our main result is satisfied in the case where $m=n$ and $m\to \infty$. This special case is of particular interest because of its connection with Zener cards \cite{diaconis1978statistical}. One may compare our result with the following result by Ottolini and Stefan for card guessing game with complete feedback, and conclude that the maximal expected payoff in the latter case is greater. 

\begin{theorem}\cite{ottolini2022guessing}
If $m=n$ and $m\to \infty$, then the card guessing game with complete feedback has maximal expected payoff
$$ m+\left(\frac{\pi}{\sqrt{2}}+o(1)\right)\sqrt{m\log m}.$$
\end{theorem}

\section*{Acknowledgments}
The author thanks Sam Spiro for helpful conversations.

\section{Proof of Theorem \ref{main-thm}}
\subsection{Modified payoff vector $(z_1,\ldots, z_{mn})$}
We fix a guessing strategy $\mathcal{G}$ and we assume $n\ge 1200\sqrt{m}$. Let $\left(g_1,\ldots, g_{mn}\right)$ denote the $[n]$-valued vector such that the Guesser guesses $g_t$ in the $t$-th round. Let $(y_1,\ldots, y_{mn})$ denote the Boolean vector whose coordinate $y_t$ is the indicator function of correct guess in the $t$-th round. Our goal is to give an upper bound on the expectation of $\sum_{i=1}^{mn} y_i$.

Throughout the section, let $v_{\le t}$ represent the truncation $(v_1,\ldots, v_{t})$ of an $mn$-dimensional vector $(v_1,\ldots, v_{mn})$ to the first $t$ values for the sake of conciseness.

For every integer $1\le k \le n$ and every integer $1\le t \le mn+1$, let $$a(k,t):=\left|\left\{ 1\le i < t : g_i = k\right\}\right|$$
be the number of times the label $k$ is guessed in the first $t-1$ rounds, and $$b(k,t):=\left|\left\{ 1\le i < t : g_i = k, y_i=1\right\}\right|$$ of them are correct. Then by definition we have \begin{equation}\label{sum}
    \sum_{k=1}^n a(k,t)= t-1
\end{equation}
and
\begin{equation}\label{sum_y}
    \sum_{k=1}^n b(k,t)= \sum_{i=1}^{t-1} y_i.
\end{equation} Because there are only $m$ cards labeled $k$, we have 
\begin{equation}\label{bound_m}
    b(k,t)\le  m.
\end{equation} 
By \cite[Lemma 3.1]{diaconis2022card}, we have \begin{equation}\label{3.1}
    \mathbb{E}\left(y_{t}\; \middle|\;  g_{\le t}, y_{\le t-1}\right)\le \frac{m-b(g_t, t)}{mn-a(g_t, t)-\sum_{i=1}^{t-1} y_i}
\end{equation}
if $a(g_t,t)<mn -\sum_{i=1}^{t-1} y_i$. 

By \cite[Lemma 3.8]{diaconis2022card}, the term $\sum_{i=1}^{t-1} y_i$ is at most $\left \lfloor \frac{1}{6}\sqrt{m} n\right \rfloor$ with high probability. For convenience, let $$Y:=\left \lfloor \frac{1}{6}\sqrt{m} n\right \rfloor$$ denote this integer for the rest of the proof.

One technical difficulty is that (\ref{3.1}) is only a one-sided inequality. We are going to use the coupling method to construct a Boolean vector $(z_1,\ldots, z_{mn})$ with an equality version of (\ref{3.1}). In the new version of (\ref{3.1}), the $\sum_{i=1}^{t-1} y_i$ part is replaced by $Y$. We are going to compare the vectors $(y_1,\ldots, y_{mn})$ and $(z_1,\ldots, z_{mn})$ in Subsection \ref{section-compare-yz}.
\begin{lemma}\label{z-construction}
    There exists a Boolean random vector $(z_1,\ldots, z_{mn})$ with the following properties. For every integer $1\le k\le n$ and every integer $1\le t\le mn+1$, define $c(k,t)$ by $$c(k,t):=\left|\left\{ 1\le i < t : g_i = k, z_i=1\right\}\right|.$$ Almost surely we have that
    \begin{enumerate}[label=(\alph*)]
        \item \label{property-a} for each $1\le k\le n$ and each $1\le t\le mn+1$, $$m- \max\left\{mn-a(k,t)-Y, 0 \right\} \le c(k,t)\le m;$$
        \item \label{property-b}for each $1\le t\le mn$, conditioned on $g_{\le t}$ and $y_{\le t}$, the coordinates of $z_{\le t}$ are mutually independent, and are independent from the vectors $g_{\le mn}$ and $y_{\le mn}$;
        \item \label{property-c} for each $1\le t\le mn$, if $a(g_t,t)< mn-Y$, then $$\mathbb{E}\left(z_t\;\middle|\; g_{\le t}, z_{\le t-1}\right) = \frac{m-c(g_t,t)}{mn-a(g_t,t)-Y};$$
        \item \label{property-d} for each $1\le t\le mn$, if $\mathbb{E}\left(y_t\;\middle|\; g_{\le t}, y_{\le t-1}\right) \le \mathbb{E}\left(z_t\;\middle|\; g_{\le t}, z_{\le t-1}\right),$ then $y_t\le z_t$.
    \end{enumerate}
\end{lemma}
\begin{proof} 
    We construct this random vector inductively. Because $n\ge 1200\sqrt{m}$, the condition \ref{property-a} holds for $t=1$. Suppose that $z_{\le t-1}$ has been defined in the first $t-1$ rounds with the conditions \ref{property-a}, \ref{property-b}, \ref{property-c} and \ref{property-d} for $t-1$. Additionally, assume that \ref{property-a} holds for $t$, because this is also a property of $z_{\le t-1}$. 
    
    In the $t$-th round, after the Guesser makes the guess $g_t$, we compare $a(g_t, t)$ with $mn-Y$. If $a(g_t,t)\ge mn-Y$, then we take $z_t=0$. Otherwise, define a real number $p_t$ as $$p_t:=\frac{m-c(g_t,t)}{mn-a(g_t,t)-Y}.$$ Because \ref{property-a} holds for $t$, we have $0\le p_t\le1$. Thus, we can define $z_t$ as a new Bernoulli random variable with 
    $$\mathbb{E}\left(z_t\;\middle|\; g_{\le t}, y_{\le t},z_{\le t-1}\right)=\begin{cases}
    \frac{(1-p_t) y_t+p_t-\mathbb{E}\left(y_t\;\middle| \; g_{\le t}, y_{\le t-1}\right)}{1-\mathbb{E}\left(y_t\;\middle| \; g_{\le t}, y_{\le t-1}\right)} &\mbox{, if } p_t \ge\mathbb{E}\left(y_t\;\middle| \; g_{\le t}, y_{\le t-1}\right),\\
    \frac{p_t y_t}{\mathbb{E}\left(y_t\;\middle| \; g_{\le t}, y_{\le t-1}\right)}& \mbox{, if }p_t < \mathbb{E}\left(y_t\;\middle| \; g_{\le t}, y_{\le t-1}\right).
\end{cases}$$
Hence \ref{property-b} holds for $t$. 

Because of the conditional independence \ref{property-b}, we have
$$\mathbb{E}\left(y_t\;\middle| \; g_{\le t}, y_{\le t-1}\right) =\mathbb{E}\left(y_t\;\middle| \; g_{\le t}, y_{\le t-1},z_{\le t-1}\right),$$
which implies 
\begin{align*}
    &\mathbb{E}\left(z_t\;\middle|\; g_{\le t}, z_{\le t-1}\right)\\
    =&\mathbb{E}\left(\mathbb{E}\left(\mathbb{E}\left(z_t\;\middle|\; g_{\le t}, y_{\le t}, z_{\le t-1}\right)\;\middle|\; g_{\le t}, y_{\le t-1}, z_{\le t-1}\right)\;\middle|\; g_{\le t}, z_{\le t-1}\right)\\
    =&\mathbb{E}\left(p_t\;\middle|\; g_{\le t}, z_{\le t-1}\right)\\
    =&p_t
\end{align*}
when $a(g_t,t)< mn-Y$. Hence \ref{property-c} holds for $t$.

In the case where $a(g_t, t)\ge mn -Y$, the equation $z_t=0$ holds almost surely. With the condition of \ref{property-d}, we have $\mathbb{E}\left(y_t\;\middle|\; g_{\le t}, y_{\le t-1}\right)=0$, which implies $y_t=0$ almost surely. Otherwise, we have $a(g_t, t)< mn -Y$. With \ref{property-c} and the condition of \ref{property-d}, we have $\mathbb{E}\left(z_t\;\middle|\; g_{\le t}, y_{\le t}, z_{\le t-1}\right)=1$ by the definition of $z_t$, which implies $z_t=1$ almost surely. Hence \ref{property-d} holds for $t$.

The last step is to prove \ref{property-a} for $t+1$. If $c(k,t)=m$, then $z_t=0$ almost surely by definition, which implies $c(k,t+1)=m$ almost surely. Hence \ref{property-a} holds for $t+1$ assuming $c(k,t)=m$. Otherwise, by inductive hypothesis, we have $a(k,t) < mn-Y$ and $c(k,t)<m$. By definition, $c(k,t)$ and $c(k,t+1)$ are integers with $$c(k,t+1)-c(k,t)\in \left\{0,1\right\},$$ and similarly $a(k,t)$ and $a(k,t+1)$ are integers with $$a(k,t+1)-a(k,t)\in \left\{0,1\right\}.$$ If \ref{property-a} does not hold for $t+1$, then we must have $$a(k,t+1)=a(k,t)+1$$ and $$c(k,t+1)=c(k,t)= m-mn +a(k,t)-Y.$$ The equation $a(k,t+1)=a(k,t)+1$ implies that $g_t=k$. Then the equation $c(k,t+1)=c(k,t)$ implies that $z_t=0$. And the equation $c(k,t)= m-mn +a(k,t)-Y$ implies that $z_t=0$ almost surely by the definition of $z_t$. Hence \ref{property-a} folds for $t+1$ almost surely. 

Therefore, we have constructed the Boolean vector $(z_1,\ldots,z_{mn})$ with desired properties by the principle of mathematical induction.

\end{proof}

\subsection{Deviation of $\sum_{i=1}^{mn} z_i$ from $m$}
Let $(z_1,\ldots,z_{mn})$ be the Boolean vector constructed as in Lemma \ref{z-construction}. For each $1\le k\le n$ and each $1\le t\le mn+1$, define $c(k,n)$ as in Lemma \ref{z-construction}. Similar to (\ref{sum}) and (\ref{sum_y}) we have \begin{equation}\label{sum_z}
    \sum_{k=1}^n c(k,t)=\sum_{i=1}^{t-1} z_i.
\end{equation}
By (\ref{sum}), we have $$\sum_{i=1}^{t-1} z_i -\frac{t-1}{n}= \sum_{k=1}^n\left(c(k,t)-\frac{a(k,t)}{n}\right).$$
Thus, to bound the deviation of $\sum_{i=1}^{mn}z_i$ from $m$, it is natural to consider the deviation of $c(k,t)$ from $\frac{a(k,t)}{n}$ for each $1\le k\le n$ and each $1\le t\le mn+1$. The term $c(k,t)-\frac{a(k,t)}{n}$ also appears in the deviation of $\mathbb{E}\left(z_t \;\middle|\; g_{\le t},z_{\le t-1}\right)$ from $\frac{1}{n}$, according to the condition \ref{property-c} in Lemma \ref{z-construction}.

Define a function $f:\mathbb{R}\to \mathbb{R}$ by $$f(x):=\begin{cases} x^2 &\mbox{, if } x\le 0,\\0&\mbox{, if } 0<x<\frac{Y}{n},\\ \left(x-\frac{Y}{n}\right)^2 &\mbox{, if } x\ge\frac{Y}{n}. \end{cases}$$
By the identity $$\left(x-\frac{Y}{n}\right)^2 = \frac{x^2}{2}   - \frac{Y^2}{n^2} +\frac{1}{2}\left(x-\frac{2Y}{n}\right)^2,$$ we have \begin{equation}\label{f-bound} f(x)\ge \max\left\{0, \frac{x^2}{2} -\frac{Y^2}{n^2}\right\}\end{equation} for all $x\in \mathbb{R}$.

For each $1\le k \le n$ and each $1\le t\le mn+1$, let $X_{k,t}$ denote the random variable 
    $$X_{k,t} := f\left(c(k,t)-\frac{a(k,t)}{n}\right) -3c(k,t) -\frac{3a(k,t)}{n}.$$
We prove that $X_{k,1}, X_{k,2},\ldots, X_{k, mn+1}$ is a supermartingale.

\begin{lemma}\label{supermartingale}
    For each $1\le k\le n$ and each $1\le t\le mn$, we have $$\mathbb{E}\left(X_{k,t+1}\;\middle| \; g_{\le t}, z_{\le t-1}\right)\le X_{k,t}.$$
\end{lemma}
\begin{proof}
    If $g_t\neq k$ or $a(k,t)\ge mn-Y$, then by the condition \ref{property-a} in Lemma \ref{z-construction}, we have $a(k,t+1)=a(k,t)$ and $c(k,t+1)=c(k,t)$. So by definition we have $X_{k,t+1}=X_{k,t}$. For the rest of the proof, suppose that $g_t=k$ and $a(k,t)< mn-Y$, then we have $a(k,t+1)=a(k,t)+1$ and $c(k,t+1)=c(k,t)+z_t$. By the condition \ref{property-c} in Lemma \ref{z-construction}, we have $$\mathbb{E}\left(z_t-\frac{1}{n}\;\middle|\; g_{\le t}, z_{\le t-1}\right) = \frac{- c(k,t)+\frac{a(k,t)}{n}+\frac{Y}{n}}{mn-a(k,t)-Y}.$$ 
    
    Let $x$ denote the term $c(k,t)-\frac{a(k,t)}{n}$. If $-1 \le x \le \frac{Y}{n}+1$, then we have $\left| f(x+\Delta)-f(x)\right| \le 3|\Delta|$ for any $|\Delta|\le 1$, so
    \begin{align*}
    &\mathbb{E}\left(X_{k,t+1}\;\middle| \; g_{\le t}, z_{\le t-1}\right)-X_{k,t}\\
    =&\mathbb{E}\left(f\left(x+z_t-\frac{1}{n}\right)-f\left(x\right)-3z_t-\frac{3}{n}\;\middle| \; g_{\le t}, z_{\le t-1} \right)\\
    \le& \mathbb{E}\left( 3\left|z_t-\frac{1}{n}\right|-3z_t-\frac{3}{n}\;\middle| \; g_{\le t}, z_{\le t-1} \right)\\
    \le&0.
    \end{align*}
    If $x<-1$, then by definition we have $f(x+\Delta) =(x+\Delta)^2$ for any $|\Delta|\le 1$, so
\begin{align*}
    &\mathbb{E}\left(X_{k,t+1}\;\middle| \; g_{\le t}, z_{\le t-1}\right)-X_{k,t}\\
    =&\mathbb{E}\left(f\left(x+z_t-\frac{1}{n}\right)-f\left(x\right)-3z_t-\frac{3}{n}\;\middle| \; g_{\le t}, z_{\le t-1} \right)\\
    =& \mathbb{E}\left( \left(x+z_t-\frac{1}{n}\right)^2-x^2-3z_t-\frac{3}{n}\;\middle| \; g_{\le t}, z_{\le t-1} \right)\\
    =& \mathbb{E}\left( 2 x \left(z_t-\frac{1}{n}\right)-\frac{2(n+1)z_n}{n}-\frac{3n-1}{n^2}\;\middle| \; g_{\le t}, z_{\le t-1} \right)\\
    \le&2x \;\mathbb{E}\left(z_t-\frac{1}{n}\;\middle|\; g_{\le t}, z_{\le t-1}\right) \\
    =& \frac{2x\left(- x+ \frac{Y}{n}\right)}{mn-a(k,t)-Y}
    \\
    \le&0.
    \end{align*}
    If $x>\frac{Y}{n}+1$, then by definition we have $f(x+\Delta) =(x+\Delta-\frac{Y}{n})^2$ for any $|\Delta|\le 1$, so
    \begin{align*}
    &\mathbb{E}\left(X_{k,t+1}\;\middle| \; g_{\le t}, z_{\le t-1}\right)-X_{k,t}\\
    =&\mathbb{E}\left(f\left(x+z_t-\frac{1}{n}\right)-f\left(x\right)-3z_t-\frac{3}{n}\;\middle| \; g_{\le t}, z_{\le t-1} \right)\\
    =& \mathbb{E}\left( \left(x-\frac{Y}{n}+z_t-\frac{1}{n}\right)^2-\left(x-\frac{Y}{n}\right)^2-3z_t-\frac{3}{n}\;\middle| \; g_{\le t}, z_{\le t-1} \right)\\
    =& \mathbb{E}\left( 2 \left(x-\frac{Y}{n}\right) \left(z_t-\frac{1}{n}\right)-\frac{2(n+1)z_n}{n}-\frac{3n-1}{n^2}\;\middle| \; g_{\le t}, z_{\le t-1} \right)\\
    \le&2\left(x-\frac{Y}{n}\right)\mathbb{E}\left(z_t-\frac{1}{n}\;\middle|\; g_{\le t}, z_{\le t-1}\right) \\
    =& -\frac{2\left(x-\frac{Y}{n}\right)^2}{mn-a(k,t)-Y}
    \\
    \le&0.
    \end{align*}
    
    Therefore, the stochastic process $X_{k,1}, X_{k,2},\ldots, X_{k, mn+1}$ is a supermartingale.
\end{proof}

Let $(w_1,\ldots, w_{mn})$ denote the Boolean vector whose coordinate $w_t$ is the indicator function of the event $a(g_t,t)\le \frac{1}{2}mn$. We are going to bound the sums $\sum_{i=1}^{mn}(1-w_i)\left(z_i-\frac{1}{n}\right)$ and $\sum_{i=1}^{mn} w_i \left(z_i-\frac{1}{n}\right)$ separately by optional stopping theorem.

\begin{lemma}\label{bound-z-1}
    We have $$\left|\mathbb{E}\left(\sum_{i=1}^{mn}\left(1-w_i\right)\left(z_i-\frac{1}{n}\right)\right)\right|\le 2\sqrt{6\sum_{i=1}^{mn} \mathbb{E}(z_i) +8m}.$$
\end{lemma}
\begin{proof}
For each $1\le k\le n$, define the integer $1\le \tau_k\le mn+1$ by
$$\tau_k:=\begin{cases}
    t &\mbox{, if } t\mbox{ is the least integer } i \mbox{ with } w_i=0,
    \\
    mn+1 &\mbox{, if } w_i=1 \mbox{ for all }i. 
\end{cases}$$
Then by (\ref{sum}), (\ref{sum_z}) and (\ref{f-bound}), we have
\begin{align*}
    &\sum_{k=1}^n \left(X_{k,\tau_k}+ X_{k, mn+1}\right)\\
    =& \sum_{k=1}^n \left(f\left(c(k,\tau_k)-\frac{a(k,\tau_k)}{n}\right)+f\left(k,mn+1)-\frac{a(k,mn+1)}{n}\right)\right) \\
    &-3 \sum_{k=1}^n\left(c(k,\tau_k) +\frac{a(k,\tau_k)}{n}+c(k,mn+1) +\frac{a(k,mn+1)}{n}\right)\\
    \ge&\max_{1\le k\le n} \left(f\left(c(k,\tau_k)-\frac{a(k,\tau_k)}{n}\right)+f\left(k,mn+1)-\frac{a(k,mn+1)}{n}\right)\right) \\
    & -6 \sum_{k=1}^n\left(c(k,mn+1) +\frac{a(k,mn+1)}{n}\right)\\
    =&\max_{1\le k\le n} \left(f\left(c(k,\tau_k)-\frac{a(k,\tau_k)}{n}\right)+f\left(k,mn+1)-\frac{a(k,mn+1)}{n}\right)\right)\\
    & -6 \sum_{i=1}^{mn} z_i -6m\\
    \ge& \frac{1}{2} \max_{1\le k\le n}\left(\left(c(k,\tau_k)-\frac{a(k,\tau_k)}{n}\right)^2 + \left(c(k,mn+1)-\frac{a(k,mn+1)}{n}\right)^2\right) \\
    & -\frac{2Y^2}{n^2}-6 \sum_{i=1}^{mn} z_i -6m\\
    \ge &\frac{1}{4}\max_{1\le k\le n} \left(c(k,mn+1)-c(k,\tau_k)-\frac{a(k,mn+1)-a(k,\tau_k)}{n}\right)^2\\
    &-6\sum_{i=1}^{mn} z_i -8m\\
    =&\frac{1}{4} \left( \sum_{i=1}^{mn} (1-w_i)\left(z_i -\frac{1}{n}\right)\right)^2-6\sum_{i=1}^{mn} z_i -8m.
\end{align*}

Because there are only $mn$ guesses in total, for any $1\le i<j\le mn$ with $w_i=w_j=1$, we have $g_i=g_j$. Hence by definition we have $\tau_k = mn+1$ for at least $n-1$ values of $k$ in $[n]$. 

Because $\tau_k$ and $mn+1$ are stopping times, by Lemma \ref{supermartingale} and optional stopping theorem, we have $$\mathbb{E}\left(X_{k, \tau_k}\right)\le \mathbb{E}\left(X_{k,1}\right)=0$$
and 
$$\mathbb{E}\left(X_{k, mn+1}\right)\le \mathbb{E}\left(X_{k,1}\right)=0$$
for each $1\le k \le n$. Therefore 
\begin{align*}
    &\left|\mathbb{E}\left(\sum_{i=1}^m (1-w_i)\left(z_i -\frac{1}{n}\right) \right)\right|\\
    \le &\sqrt{\mathbb{E}\left(\left(\sum_{i=1}^m (1-w_i)\left(z_i -\frac{1}{n}\right) \right)^2\right)}\\
    = &\sqrt{\mathbb{E}\left(\max_{1\le k\le n} \left(c(k,mn+1)-c(k,\tau_k)-\frac{a(k,mn+1)-a(k,\tau_k)}{n}\right)^2\right)}\\
    \le& \sqrt{\mathbb{E}\left( 4\sum_{k=1}^n \left(X_{k,\tau_k}+X_{k,mn+1}\right) +24\sum_{i=1}^{mn} z_i +32m\right)}\\
    \le& 2\sqrt{6\sum_{i=1}^{mn} \mathbb{E}(z_i) +8m}.
\end{align*}
\end{proof}

An upper bound on $\left|\mathbb{E} \left(w_t\left(z_t-\frac{1}{n}\right)\right) \right|$ can be established in a similar manner.

\begin{lemma}\label{bound-z-2}
    For each $1\le t\le mn$, we have 
    $$\left|\mathbb{E} \left(w_t\left(z_t-\frac{1}{n}\right)\right) \right|\le \frac{4}{mn}\sqrt{6\sum_{i=1}^{mn} \mathbb{E}(z_i) +8m}.$$
\end{lemma}
\begin{proof}
By (\ref{sum}), (\ref{sum_z}) and (\ref{f-bound}), we have
\begin{align*}
&\sum_{k=1}^n X_{k,t}\\
=&\sum_{k=1}^n\left(f\left(c(k,t)-\frac{a(k,t)}{n}\right) -3c(k,t) -\frac{3a(k,t)}{n} \right) \\
=&\sum_{k=1}^n f\left(c(k,t)-\frac{a(k,t)}{n}\right) -3\sum_{i=1}^{t-1} z_i-\frac{3(t-1)}{n}\\
\ge& f\left(c(g_t,t)-\frac{a(g_t,t)}{n}\right) -3\sum_{i=1}^{mn} z_i-3m\\
\ge& \frac{1}{2}\left(c(g_t,t)-\frac{a(g_t,t)}{n}\right)^2 -\frac{Y^2}{n^2}-3\sum_{i=1}^{mn} z_i-3m\\
\ge&\frac{1}{2}\left(c(g_t,t)-\frac{a(g_t,t)}{n}\right)^2 -3\sum_{i=1}^{mn} z_i-4m
\end{align*}

Because $t$ is a stopping time, by Lemma \ref{supermartingale} and optional stopping theorem, we have 
$$\mathbb{E}(X_{k,t})\le 0$$
for each $1\le k\le n $ and each $1\le t\le mn+1$. 

When $w_t=1$, we have $$a(g_t,t)\le \frac{1}{2}mn\le \frac{2}{3}mn-Y.$$ Therefore, by the condition \ref{property-c} in Lemma \ref{z-construction}, we have 
\begin{align*}
    &\left|\mathbb{E}\left(w_t \left( z_t-\frac{1}{n}\right)\right)\right|\\
    =&\left|\mathbb{E}\left(\mathbb{E}\left(w_t \left( z_t-\frac{1}{n}\right)\; \middle|\; g_{\le t},z_{\le t-1}\right)\right)\right|\\
    \le&\mathbb{E}\left(\left|\mathbb{E}\left(w_t \left( z_t-\frac{1}{n}\right)\; \middle|\; g_{\le t},z_{\le t-1}\right)\right|\right)\\
    \le &\mathbb{E}\left(\frac{3\left|-c(g_t,t)+\frac{a(g_t,t)}{n}+\frac{Y}{n}\right|}{mn}\right)\\
    \le &\frac{3}{mn}\mathbb{E}\left(\left|c(g_t,t)-\frac{a(g_t,t)}{n}\right|\right)+\frac{1}{2\sqrt{m} n}\\
    \le& \frac{3}{mn}\sqrt{\mathbb{E}\left(\left(c(g_t,t)-\frac{a(g_t,t)}{n}\right)^2\right)}+\frac{1}{2\sqrt{m}n}\\
    \le& \frac{3}{mn}\sqrt{\mathbb{E}\left(2\sum_{k=1}^n X_{k,t} +6\sum_{i=1}^{mn} z_i +8m\right)}+\frac{1}{2\sqrt{m}n}\\
    \le& \frac{3}{mn}\sqrt{6\sum_{i=1}^{mn} \mathbb{E}(z_i) +8m}+\frac{1}{2\sqrt{m}n}\\
    \le& \frac{4}{mn}\sqrt{6\sum_{i=1}^{mn} \mathbb{E}(z_i) +8m}.
\end{align*}
\end{proof}

With Lemma \ref{bound-z-1} and Lemma \ref{bound-z-2}, we can bound the deviation of the expectation of $\sum_{i=1}^{mn} z_i$ from $m$.

\begin{lemma}\label{bound-z-all}
    We have $$\left|\sum_{i=1}^{mn}\mathbb{E}(z_i) -m\right|\le 300\sqrt{m}.$$
\end{lemma}
\begin{proof}
    By Lemma \ref{bound-z-1} and Lemma \ref{bound-z-2}, we have 
    \begin{align*}
        &\left|\sum_{i=1}^{mn}\mathbb{E}(z_i) -m\right|\\
        \le &\left|\mathbb{E}\left(\sum_{i=1}^{mn}\left(1-w_i\right)\left(z_i-\frac{1}{n}\right)\right)\right|+ \sum_{i=1}^{mn}\left|\mathbb{E} \left(w_i\left(z_i-\frac{1}{n}\right)\right) \right|\\
        \le& 6\sqrt{6\sum_{i=1}^{mn} \mathbb{E}(z_i) +8m}.
    \end{align*}
Therefore we have 
$$\left|\sum_{i=1}^{mn}\mathbb{E}(z_i) -m\right|\le 300\sqrt{m}.$$
\end{proof}

\subsection{Upper bound on $\sum_{i=1}^{mn} y_i$}\label{section-compare-yz}

By (\ref{sum_y}) and (\ref{sum_z}), we have 
$$\sum_{i=1}^{t-1} y_i -\sum_{i=1}^{t-1}z_i =\sum_{k=1}^n\left(b(k,t)-c(k,t)\right).$$
In order to compare $\sum_{i=1}^{mn} y_i$ with $\sum_{i=1}^{mn}z_i$, we only need to compare $b(k,mn+1)$ with $c(k,mn+1)$ for each $1\le k\le n$. We prove the following statement inductively.

\begin{lemma}\label{compare}
    For every $1\le k\le n$ and $1\le t\le mn+1$, then we have $$ \left(b(k,t)\le c(k,t)\right)\lor \left(\sum_{i=1}^{mn} y_i> Y\right)$$
    almost surely.
\end{lemma}
\begin{proof}
We induct on $t$. The statement automatically holds for $t=1$. Suppose that the statement holds for $t$ ($1\le t\le mn$). For the sake of contradiction, suppose that $b(k,t+1)> c(k,t)$ and $\sum_{i=1}^{mn} y_i \le Y$.

Because $b(k,t)$ and $b(k,t+1)$ are integers with $$b(k,t+1)-b(k,t)\in \left\{0,1\right\},$$ and $c(k,t)$ and $c(k,t+1)$ are integers with $$c(k,t+1)-c(k,t)\in \left\{0,1\right\},$$
by inductive hypothesis, we have $$b(k,t+1)-1=b(k,t)=c(k,t)=c(k,t+1).$$
The equation $b(k,t+1)=b(k,t)+1$ implies that $g_t=k$ and $y_t=1$. Then the equation $c(k,t)=c(k,t+1)$ implies that $z_t=0$. 

By (\ref{bound_m}), we have $b(k,t+1)\le m$, so $c(k,t+1)\le m-1$. Thus by the condition \ref{property-a} in Lemma \ref{z-construction}, we have $a(k,t)\le a(k,t+1) < mn -Y$. By the condition \ref{property-c} in Lemma \ref{z-construction}, we have $$\mathbb{E}\left(z_t\;\middle|\; g_{\le t}, z_{\le t-1}\right) = \frac{m-c(g_t,t)}{mn-a(g_t,t)-Y}=\frac{m-b(g_t,t)}{mn-a(g_t,t)-Y}.$$
By $\sum_{i=1}^{mn} y_i\le Y$ and (\ref{3.1}), we have $$\mathbb{E}\left(y_{t}\; \middle|\;  g_{\le t}, y_{\le t-1}\right)\le \mathbb{E}\left(z_{t}\; \middle|\;  g_{\le t}, z_{\le t-1}\right).$$
By the condition  \ref{property-d} in Lemma \ref{z-construction}, we have $y_t\le z_t$ almost surely, which is not compatible with the conditions $y_t=1$ and $z_t=0$. Therefore the statement holds for $t+1$.

By the principle of mathematical induction, the statement hold for every $t$.
\end{proof}

Finally, we prove our main result by combining Lemma \ref{bound-z-all}, Lemma \ref{compare} and \cite[Lemma 3.8]{diaconis2022card}.

\begin{proof}[Proof of Theorem \ref{main-thm}]
    Because $n\ge 1200\sqrt{m}$, by \cite[Lemma 3.8]{diaconis2022card}, we have 
\begin{equation}
    \mathbb{P}\left(\sum_{i=1}^{m n} y_i>Y\right)\le 2 e^{-\sqrt{m} n/72}\le \frac{144}{\sqrt{m}n}.
\end{equation}
    By Lemma \ref{bound-z-all} and Lemma \ref{compare}, we have
    \begin{align*}
        &\sum_{i=1}^{mn}\mathbb{E}\left(y_i\right)\\
        =&\sum_{i=1}^{mn}\mathbb{E}\left(y_i\;\mathds{1}\left(\sum_{i=1}^{mn}y_i\le Y \right)\right) +\sum_{i=1}^{mn}\mathbb{E}\left(y_i\;\mathds{1}\left(\sum_{i=1}^{mn}y_i>Y \right)\right)\\
        \le &\sum_{i=1}^{mn}\mathbb{E}\left(y_i\;\mathds{1}\left(\sum_{i=1}^{mn}y_i\le Y \right)\right) +mn \;\mathbb{P}\left(\sum_{i=1}^{mn}y_i>Y\right)\\
        \le &\sum_{i=1}^{mn}\mathbb{E}\left(y_i\;\mathds{1}\left(\sum_{i=1}^{mn}y_i\le Y \right)\right) +144\sqrt{m}\\
        =&\sum_{i=1}^{mn}\sum_{k=1}^n\mathbb{E}\left(b(k,mn+1)\;\mathds{1}\left(\sum_{i=1}^{mn}y_i\le Y \right)\right)+144\sqrt{m}\\
        \le &\sum_{i=1}^{mn}\sum_{k=1}^n\mathbb{E}\left(c(k,mn+1)\;\mathds{1}\left(\sum_{i=1}^{mn}y_i\le Y \right)\right)+144\sqrt{m}\\
        =&\sum_{i=1}^{mn}\mathbb{E}\left(z_i\;\mathds{1}\left(\sum_{i=1}^{mn}y_i\le Y \right)\right)+144\sqrt{m}\\
        \le&\sum_{i=1}^{mn}\mathbb{E}\left(z_i\right)+144\sqrt{m}\\
        \le&\left|\sum_{i=1}^{mn}\mathbb{E}\left(z_i\right)-m\right|+m+144\sqrt{m}\\
        \le& m+500\sqrt{m}.
    \end{align*}
\end{proof}

\end{document}